\documentclass[12pt, article]{amsart} 
\usepackage{setspace}
\usepackage{amsfonts}
\usepackage[shortlabels]{enumitem}
\usepackage{amsthm}
\usepackage{amsmath}
\usepackage{amssymb} %
\usepackage{slashed}
\usepackage{amscd}
\usepackage[latin2]{inputenc}
\usepackage{t1enc}
\usepackage[mathscr]{eucal}
\usepackage{indentfirst}
\usepackage{graphicx}
\usepackage{graphics}
\usepackage{pict2e}
\usepackage{epic}
\numberwithin{equation}{section}
\usepackage[margin=3.2cm]{geometry}
\usepackage{epstopdf} 

\usepackage{tikz-cd}
\usepackage{dsfont}
\usetikzlibrary{matrix}

\newcommand\CC{{\mathbb C}} 
\newcommand\CP{{\mathbb C}\mathbb{P}} 
\newcommand\QQ{{\mathbb Q}} 
 
\newcommand{\fd}{\mathop{\mathrm{fd}}} 

\allowdisplaybreaks 

\theoremstyle{plain}
\newtheorem{theorem}{Theorem}[section]
\newtheorem*{theorem*}{Theorem}
\newtheorem{lemma}[theorem]{Lemma}

\newtheorem{proposition}[theorem]{Proposition}

 \theoremstyle{definition}

\newtheorem{remark}[theorem]{Remark}
\newtheorem{?}[theorem]{Problem}

\newtheorem*{Acknowledgments*}{Acknowledgments}

\newcommand{\im}{\operatorname{im}}

\begin{document}

\title{Verifying the Hilali Conjecture up to formal dimension twenty}

\author{Spencer Cattalani}
\author{Aleksandar Milivojevi\'{c}}
\address{Department of Mathematics, Stony Brook University, NY 11794}\email{spencer.cattalani@stonybrook.edu}\urladdr{}\email{aleksandar.milivojevic@stonybrook.edu}\urladdr{}

\begin{abstract} We prove that in formal dimension $\leq 20$ the Hilali conjecture holds, i.e. that the total dimension of the rational homology bounds from above the total dimension of the rational homotopy for a simply connected rationally elliptic space. \end{abstract}

\maketitle

\section{Introduction}

The Hilali conjecture [HM08a] in rational homotopy theory states that for a minimal commutative differential graded algebra over the rationals $(\Lambda V, d)$ with $V^1 = 0$ whose cohomology $H^*(\Lambda V, d) = \bigoplus_i H^i(\Lambda V, d)$ and space of indecomposables $V$ are both finite--dimensional, we have $H^*(\Lambda V, d) \geq \dim V$. Translated into a geometric statement, this says that the total dimension of the rational cohomology of a simply connected space bounds the total dimension of the rational homotopy from above if the latter quantity is finite.

Simply connected spaces with such minimal models, called rationally elliptic spaces, are known to satisfy very restrictive topological conditions. For such a space $X$, the topological Euler characteristic is non-negative and the homotopy Euler characteristic $\sum_i (-1)^i \pi_i(X)\otimes \QQ$ is non-positive; furthermore, one is non-zero if and only if the other is zero [FHT, Prop. 32.10]. Such spaces are akin to closed manifolds, as they satisfy a Poincar\'e duality on their rational cohomology [FHT Prop. 38.3]: $H^*(X;\QQ) \cong H^{n-*}(X;\QQ)$, where $n$ is the formal dimension $\fd(X)$ of $X$, i.e. the largest index for which the rational cohomology does not vanish. In fact, if the homotopy Euler characteristic of $X$ is negative, one can find a simply connected closed smooth manifold $M$ and a rational homotopy equivalence $M \to X$ by the Barge--Sullivan theorem [FrH79, p.124].

Friedlander and Halperin [FrH79] identified the condition under which a set of integers occurs as the degrees of a homogeneous basis of $\pi_*(X)\otimes \QQ$ of a rationally elliptic space $X$. Namely, the sequence $(2a_1, \ldots, 2a_r : 2b_1 - 1, \ldots, 2b_q -1)$ denotes the degrees of a homogeneous basis of $\pi_*(X)\otimes \QQ$ of some elliptic $X$ if and only if the following \textit{strong arithmetic condition} is satisfied: for every subsequence $A^*$ of $(a_1, \ldots, a_r)$ of length $s$, at least $s$ many elements $b_j$ in $(b_1, \ldots, b_q)$ can be written as $b_j = \sum_{a_i \in A^*} \gamma_{ij}a_i$, where the $\gamma_{ij}$ are non-negative integers whose sum for any fixed $j$ is at least two. Call such a sequence a \textit{homotopy rank type}; note that the homotopy rank type does not uniquely determine the space $X$ up to rational homotopy equivalence, even amongst elliptic spaces.

Using this characterization, Nakamura and Yamaguchi [NaYa11] wrote a \texttt{C++} program to output all the homotopy rank types of simply connected elliptic spaces up to a given formal dimension. In the present paper, after establishing some preliminary results, we will verify the Hilali conjecture up to formal dimension 20 by employing our results into the code of [NaYa11] to significantly reduce the number of homotopy rank types that need to be considered manually. In [HM08b], the conjecture is claimed to be verified up to formal dimension 10; in [NaYa11] this claim was pushed to formal dimension 16. However, the tables of homotopy rank types in [HM08b] are slightly incomplete (for example the homotopy rank type $(2:11)$ corresponding to $\CP^5$ is not present in Table 1 therein), and the current authors failed to understand how an inequality in the proof of the crucial Proposition 4.3 in the latter article was obtained. We hence reverify the conjecture in these dimensions carefully and extend the verification up to dimension 20. In the next section the reader may see how the number of homotopy rank types increases considerably with the formal dimension.

Throughout, $(\Lambda V, d)$ will denote a minimal commutative differential graded algebra modelling a given space $X$; $V^k$ will denote the degree $k$ elements of the space of indecomposables $V$, and $(\Lambda V)^k$ the degree $k$ elements in the algebra. Likewise $\Lambda V^{\leq m}$ will denote the subalgebra of $\Lambda V$ generated by the elements of degree at most $m$, and $(\Lambda V^{\leq m})^k$ will denote the vector space of degree $k$ elements in this subalgebra. For ease of notation we will denote by $H^*$ the total cohomology $\bigoplus_i H^i(\Lambda V, d)$. 

\begin{Acknowledgments*}

The authors would like to acknowledge the support of the Directed Reading Program at Stony Brook University, under which this project was initiated. Several computations were carried out with the Commutative Differential Graded Algebras module for SageMath [Sage] written by Miguel Marco and John Palmieri, for which the authors are duly grateful. We thank the referee for their help in streamlining the exposition.

\end{Acknowledgments*}

\section{Verification in dimension $\leq 20$}

We now collect some general statements and ad hoc arguments which we will implement into the code found in [NaYa11] in order to reduce the verification of the Hilali conjecture in formal dimension $\leq 20$ to several cases, which we will then rule out by hand. Following the notation of [NaYa11], homotopy rank types will be denoted by $(2a_1, \ldots, 2a_n : 2b_1 - 1, \ldots, 2b_{n+p}-1)$, where the sequences $a_i$ and $b_i$ are (not necessarily strictly) increasing. Note that $-p$ equals the homotopy Euler characteristic of any space $X$ realizing the given homotopy rank type. 

\begin{proposition} If $p=0$, then the Hilali conjecture holds. \end{proposition}

\begin{proof}  The vanishing of the homotopy Euler characteristic $\chi_\pi$ implies that the Euler characteristic of any such space is positive. This now implies the space admits a pure minimal model (the existence of a pure model is stated in [FHT Prop. 32.10], and minimality of this model can be seen from the proof therein), and so by [BFMM14 Section 3] the conjecture holds. \end{proof}

\begin{remark} In the lemmas to follow we will rely on the existence of elements of $V$ in degree strictly smaller than half the formal dimension. We thus verify now that the Hilali conjecture holds for simply connected spaces $X$ of formal dimension $n$ for which $b_1, \ldots, b_{\left \lceil{\frac{n}{2}}\right \rceil - 1} = 0$. If the formal dimension is odd or if $b_{\tfrac{n}{2}} = 0$, then by Poincar\'e duality $X$ is rationally homotopy equivalent to a sphere, for which the conjecture holds. If the formal dimension is even, $n=2k$,  and $\dim V_{k} = 1$, then $X$ has minimal model $\Lambda(x_k, y_{3k-1})$ with $dx = 0$, $dy = x^3$, and so the conjecture holds. If $\dim V_k = 2$, the space $X$ will admit a minimal model over the complex numbers of the form $\Lambda(x_k, x'_k, y_{2k-1}, y'_{2k-1})$ with $dx = dx' = 0$ and $dy = x^2 - x'^2$, $dy' = xx'$ (tensoring with the complex numbers has the advantage of making the nondegenerate pairing in the middle degree cohomology equivalent to the pairing represented by the identity matrix). We see that $\dim H^*(X;\CC) = 4$ and $\dim \pi_*(X) \otimes \CC = 4$; since these dimensions are independent of the choice of coefficient field of characteristic zero, the conjecture is verified. In the case of $\dim V_k \geq 3$, one can build the minimal model over the complex numbers (again to simplify the intersection pairing) and see that one must introduce at least two generators in degrees $>n$, showing that this space is not elliptic [FHT p.441] (cf. with the rational hyperbolicity of $\#_{i=1}^k \CP^2$ for $k\geq 3$). Alternatively, any rational Poincar\'e duality space with $b_1, \ldots, b_{\left \lceil{\frac{n}{2}}\right \rceil - 1}= 0$ is formal by [Mi79] and hence satisfies the Hilali conjecture by [HiMa08a, Theorem 2] if it is rationally elliptic. \end{remark}

\begin{lemma} Let $X$ be a simply connected rationally elliptic space with $p>0$. Suppose the smallest degree in which $\pi_*(X) \otimes \QQ$ is nonzero is strictly less than $\tfrac{\fd(X)}{2}$, and denote the dimension of this space by $k$. If $\fd(X)$ is odd, then $\dim H^*(X;\QQ) \geq 2k+2$. If $\fd(X)$ is even, and the smallest degree in which $\pi_*(X) \otimes \QQ$ is nonzero is odd, then $\dim H^*(X;\QQ) \geq 4k$. Otherwise, if the smallest nonzero degree of $\pi_*(X)\otimes \QQ$ is even, we have $\dim H^*(X;\QQ) \geq 4k+4$. \end{lemma}

\begin{proof} Note that every element in the smallest nonzero degree of $\pi_*(X)\otimes \QQ$ corresponds to a closed, non-exact element in the minimal model of $X$ for degree reasons. The first statement now follows from $\dim H^0(X;\QQ) = 1$ and Poincar\'e duality. If the formal dimension of $X$ is even, and the smallest nonzero degree of $\pi_*(X)\otimes \QQ$ is odd, Poincar\'e duality ensures $2k$ independent cohomology classes of odd degree in $X$. Since $p\neq 0$, the Euler characteristic of $X$ is zero, providing us with another $2k$ independent cohomology classes, of even degree. If $\fd(X)$ is even and the smallest nonzero degree of $\pi_*(X)\otimes \QQ$ is even, then Poincar\'e duality gives us at least $2k+2$ independent cohomology classes in even degree, since $\dim H^0(X;\QQ) = 1$. The vanishing of the Euler characteristic then provides another $2k+2$ independent cohomology classes, now of odd degree. \end{proof}

\begin{lemma} Let $X$ be a simply connected rationally elliptic space with $p>0$. Suppose the smallest degree $d$ in which $\pi_*(X)\otimes \QQ$ is nonzero is even, and denote the dimension of this space by $k$. Suppose further that the second smallest nonzero degree of $\pi_*(X)\otimes \QQ$ is $2d-1$, of dimension $l$, with $2d-1 < \tfrac{\fd(X)}{2} - 1$. Then if $\fd(X)$ is even and $\binom{k+1}{2} \geq l$, we have $\dim H^* \geq 4(1+k + \binom{k+1}{2} - l)$; if $\binom{k+1}{2} < l$, then $\dim H^* \geq 4 \, \mathrm{max}(l - \binom{k+1}{2}, 1+k)$. If $\fd(X)$ is odd, then in either case $\dim H^* \geq 2(1+k+\left| l - \binom{k+1}{2} \right |)$. \end{lemma}

\begin{proof} We note that $(\Lambda V^{\leq d})^{2d}$ has dimension $k + \binom{k}{2}$ (spanned by squares of a basis of generators in degree $d$ and products of two distinct generators). These elements are closed, and the dimension of the image of $d$ in this space is bounded by $l$. Now the lemma follows by combining this with $\dim H^0(X;\QQ) = 1$, Poincar\'e duality, and $\chi(X) = 0$ as in Lemma 2.3.\end{proof}

\begin{remark} In the above Lemma 2.4, if the second smallest nonzero degree of rational homotopy is odd and strictly less than $2d-1$, then the corresponding elements in the minimal model are closed and non-exact, and so by Poincar\'e duality we have $\dim H^*(X;\QQ) \geq 2(1+k+l)$. If the degree is strictly greater than $2d-1$, then the inequalities in the statement of the Lemma hold with $l=0$. \end{remark}

\begin{lemma}  Let $X$ be a simply connected rationally elliptic space. Suppose the smallest degree $d$ in which $\pi_*(X)\otimes \QQ$ is nonzero is even, and denote the dimension of this space by $k$. Suppose further that the second smallest nonzero degree of $\pi_*(X)\otimes \QQ$ is $2d-1$. Denote $l = \dim \pi_{2d-1}(X) \otimes \QQ$ and $m = \dim \pi_{3d-2}(X)\otimes \QQ$. If $3d-1 < \tfrac{\fd(X)}{2}$, then $$\dim H^*(X;\QQ) \geq 2(1+k+ \left| l - \binom{k+1}{2} \right | + \mathrm{max}(0, kl-k^2 - \binom{k}{3} - m)).$$\end{lemma}

\begin{proof} Note that $\dim (\Lambda V^{\leq d})^{3d} = k^2 + \binom{k}{3}$. In $(\Lambda V)^{3d-1}$, there is a $kl$ dimensional subspace $W$ spanned by products of degree $d$ generators and degree $2d-1$ generators. The image of $d$ applied to this subspace $W$ lies in $(\Lambda V^{\leq d})^{3d}$. Since $W$ is spanned by quadratic elements, an element in it is exact only if it is in the image of the differential applied to the $m$--dimensional $V^{3d-2}$. Hence we have at least $\mathrm{max}(0, kl-k^2 - \binom{k}{3} - m)$ independent cohomology classes in degree $3d-1$. Combining this with the degree 0 class, the $k$-dimensional cohomology we obtain in degree $d$, and the $\left| l - \binom{k+1}{2} \right|$-dimensional cohomology in degree $2d-1$ or $2d$ as in Lemma 2.4, along with Poincar\'e duality, we obtain the desired bound. \end{proof}

\begin{remark} Note that if the smallest nonzero degree $d$ of $\pi_*(X)\otimes \QQ$ is odd, of dimension $l$, and the smallest nonzero even degree $d'$ of $\pi_*(X)\otimes \QQ$ is strictly less than $3d-1$, of dimension $s$, then these two vector spaces must correspond to closed non-exact elements in the minimal model of $X$ for degree reasons. Indeed, the differential applied to a generator in the smallest even degree would have to land in the subalgebra of odd degree elements, producing a polynomial all of whose monomials are at least cubic and hence of degree at least $3d$. If furthermore we denote $m = \dim \pi_{2d-1}(X) \otimes \QQ$, we have an additional $\left| \binom{l}{2} - m \right|$ independent cohomology classes in degree $2d-1$ or $2d$. Indeed, the differential applied to a degree $2d-1$ generator must land in the subspace of quadratic polynomials in the degree $d$ generators for degree reasons, which is of dimension $\binom{l}{2}$. If $2d$ and $d'$ are both strictly less than $\tfrac{\fd(X)}{2}$, then Poincar\'e duality gives us $\dim H^*(X;\QQ) \geq 2(1 + l + s + \left| \binom{l}{2} - m \right|)$. 
\end{remark}

\begin{remark} Three more quick observations that will rule out several homotopy rank types each are the following: \begin{enumerate} \item Every even generator whose degree is smaller than the lowest degree among odd generators is closed and non-exact; likewise all products of such generators (for our purposes we will only need squares) whose total degree is smaller than the lowest odd degree are closed and non-exact. \item A homotopy rank type in $\fd \geq 9$ of the form $(2,a: 3, b, c)$, where $\fd - 2 > a \geq 4$ and $\fd - 2 >  b \geq 5$, satisfies the conjecture. Let $(x,x',y,z,z')$ be generators of the corresponding degrees. Note that $x$ is closed and non-exact, and so by Poincar\'e duality, since $\dim H^0 = 1$, we have $\dim H^* \geq 4$. If we find one more independent cohomology class the conjecture is verified. If $a < b$, then $dx' = \alpha x^k y$ for some $\alpha \in \QQ$, $k \geq 1$. Now, either $y$ is closed and we are done, or $dy = \beta x^2$ for some $\beta \neq 0$; however, this would mean $dx'$ is not closed, which cannot be. If $b< a$, then $dz = \alpha x^k$ for some $\alpha \in \QQ$, $k\geq 3$. Either $y$ or $z$ is closed and we are done, or $z$ plus a multiple of $x^{k-2}y$ is closed and necessarily non-exact by minimality.  \item A homotopy rank type of the form $(2,4,a:3,3,b,c)$, where $a \geq 4$ and $b \geq 7$  in $\fd \geq 13$ satisfies the conjecture. Indeed, let $(x,z,u,y,y',v,v')$ be generators of the corresponding degrees. If $dy = dy' = 0$, we have $\dim H^* \geq 8$, so we may assume upon a change of basis that $dy = x^2$ and $dy' = 0$. Then the kernel of $d$ in degree 5 is spanned by $xy'$. If $a=4$, we see there must be a non-zero closed degree 4 generator, and $\dim H^* \geq 8$. Otherwise, if $a \geq 6$, we may assume $dz = xy'$. Then the Massey product $[xz - yy']$ is non-zero and $\dim H^* \geq 8$. 
\end{enumerate} \end{remark}

We now list the homotopy rank types remaining upon implementation of the above observations into the code of [NaYa11], and for illustration include the total number of homotopy rank types in a given formal dimension. Recall that we adopt the convention that we list the subsequences of even and odd numbers in ascending order in a given homotopy rank type. {\tiny \begin{align*} \fd \leq 14 &: \textrm{ total number of homotopy rank types }=229 \textrm{, all ruled out} \\ \fd=15 &:  \textrm{ number of homotopy rank types }=58 \\ p&=1 : (2,4,4 : 3,5,7,7), (2,2,4,4 : 3,3,3,7,7), \\ \fd=16 &:  \textrm{ number of homotopy rank types }=134 \textrm{, all ruled out}\\ \fd=17 &:  \textrm{ number of homotopy rank types }=103 \\ p&=1 : (2,4,4 : 3,7,7,7), (2,4,6 : 3,5,7,11), \\ &(2,2,4,4 : 3,3,5,7,7), (2,2,4,6 : 3,3,3,7,11), (2,4,4,4 : 3,3,7,7,7),  \\ p&=3 : (2 : 3,5,5,5), \\ \fd=18 &:  \textrm{ number of homotopy rank types }=217 \textrm{, all ruled out}\\ \fd=19 &:  \textrm{ number of homotopy rank types }=173 \\ p&=1 : (8,8 : 3,15,15), (2,4,4 : 3,5,7,11), (2,4,4 : 3,7,7,9), (2,4,6 : 3,5,9,11), \\ &(2,4,6 : 3,7,7,11), (2,4,8 : 3,5,7,15), (2,6,6 : 3,5,11,11), (4,6,6 : 3,7,11,11),  \\ &(2,2,4,4 : 3,3,3,7,11), (2,2,4,4 : 3,3,7,7,7), (2,2,4,6 : 3,3,3,9,11), \\  &(2,2,4,6 : 3,3,5,7,11), (2,2,4,8 : 3,3,3,7,15), (2,4,4,4 : 3,5,7,7,7), (2,4,4,6 : 3,3,7,7,11), \\ &(2,2,4,4,4 : 3,3,3,7,7,7), \\ p&=3 : (2 : 3,5,5,7), (2,4 : 3,3,5,5,7), \\ \fd=20 &:  \textrm{ number of homotopy rank types }=373 \\ p&=2 : (2,4,4,4 : 3,3,3,7,7,7).
\end{align*} }%

\begin{theorem} The Hilali conjecture holds in formal dimension $\leq 20$. \end{theorem}

\begin{proof} We now deal with the remaining cases listed above. When counting arguments fail to rule out a given case, we instead detect Massey products to obtain the sought after amount of cohomology. Throughout, $(\Lambda V, d)$ will denote an arbitrary minimal cdga realizing a given homotopy rank type.


In formal dimension 15, we rule out $(2,4,4:3,5,7,7)$ by noting that if the generator in degree 3 is closed, we are done as the square of the degree 2 generator is then non-exact and so $\dim H^* \geq  8$. Otherwise, we have $\ker d \cap (\Lambda V^{\leq 3})^5 = \{0\}$, and so $\dim H^4 = 2$ and $\dim H^* \geq 8$. This also rules out $(2,4,4:3,7,7,7)$, $(2,4,4:3,5,7,11)$, and $(2,4,4:3,7,7,9)$.

The remaining homotopy rank type $(2,2,4,4:3,3,3,7,7)$ in dimension 15, along with $(2,2,4,4:3,3,3,7,11)$ in dimension 19, is verified as follows. If the kernel of $V^3 \overset{d}{\rightarrow} (\Lambda V)^4$ is non-trivial, then $\dim H^3 \geq 1$ and $\dim H^4 \geq 1$ (since $\dim V^3 = \dim (\Lambda V^{\leq 2})^4$) so we are done. If the kernel of $V^3 \overset{d}{\rightarrow} (\Lambda V)^4$ is trivial, we can choose bases $\{x,x'\}$, $\{y,y',y''\}$ of $V^2$ and $V^3$ respectively such that $dy = x^2$, $dy' = x'^2$, $dy'' = xx'$. Now $\ker d \cap (\Lambda V)^5$ is spanned by $xy' - x'y''$ and $x'y - xy''$. If $V^4 \overset{d}{\rightarrow} (\Lambda V)^5$ is not injective, then we are done as $\dim H^4 + \dim H^5 \geq 2$. If this $d$ is injective, we can choose a basis $\{z,z'\}$ of $V^4$ such that $dz = xy' - x'y''$, $dz' = x'y - xy''$.  We then have the Massey products $[y'y'' + x'z]$, $[yy'' + xz']$, $[yy' - xz + x'z']$ forming a basis for $H^6(\Lambda V, d)$. We can also rule out $(2,2,4,6:3,3,3,7,11)$, $(2,2,4,6:3,3,3,9,11)$, and $(2,2,4,8:3,3,3,7,15)$ with this argument by adapting the last two sentences: Let $\{z\}$ now be a basis for $V^4$; note $\dim H^5 \geq 1$. If $dz = xy'' - x'y$, then $[y'y'' - xz] \neq 0$ gives $\dim H^6 \geq 1$; if $dz = p(xy' - x'y'') + q(xy'' - x'y)$ for some non-zero $p \in \QQ$ and $q \in \QQ$, then $[-\tfrac{q}{p} yy' - \tfrac{q^2}{p^2} yy'' + y'y'' + \tfrac{q}{p^2} xz + \tfrac{1}{p} x'z] \neq 0$. In any case, $\dim H^* \geq 10$.


On to formal dimension 17, consider $(2,4,6:3,5,7,11)$: label the generator in degree $i$ by $x_i$. If $dx_3 = 0$, we are done; so suppose that, upon rescaling, we have $dx_3 = x_2^2$, and hence $x_4$ is closed. Now, $dx_5 = ax_2x_4 + bx_2^3$ for some $a,b \in \QQ$. We see that $\ker d \cap (\Lambda V^{\leq 5})^7$ is spanned by $x_2x_5 - ax_3x_4 - bx_2^2x_3$. (Note that this verifies $\dim H^* \geq 8$ for $(2,4,8:3,5,7,15)$ in dimension 19, since there this element containing a quadratic term cannot be exact as $V^6 = \{0\}$.) Now, either $x_6$ is closed and we have $\dim H^* \geq 8$, or upon rescaling $dx_6 = x_2x_5 - ax_3x_4 - bx_2^2x_3$. Since $d(x_2x_6) = d(x_3x_5) = x_2^2x_5 - ax_2x_3x_4 - bx_2^3x_3$, we have that $\ker d \cap (\Lambda V)^8$ is spanned by $\{x_2^4, x_2^2x_4, x_4^2, x_2x_6 - x_3x_5 \}$. The vector space $(\Lambda V)^7$ is spanned by $\{x_2^2x_3, x_3x_4, x_2x_5, x_7\}$, with the image of the differential on the first three vectors being two dimensional. We conclude that $\dim H^8 \geq 1$ and thus $\dim H^* \geq 8$. We draw the same conclusion for $(2,4,6:3,5,9,11)$ in dimension 19.

For the homotopy rank types $(2,2,4,4:3,3,5,7,7)$ and $(2,2,4,4:3,3,7,7,7)$, note that if $d$ is not injective on $V^3$, we are done as $\dim H^4 \geq 2$. If it is injective, then inspection of a matrix for $(\Lambda V^{\leq 3})^5 \overset{d}{\rightarrow} (\Lambda V^{\leq 2})^6$ yields $\dim \ker d \cap (\Lambda V^{\leq 3})^5 \leq 1$, and so $\dim \ker d \cap  V^4 \geq 1$, giving us $\dim H^4 \geq 2$.

For $(2,4,4,4:3,3,7,7,7)$, we have $\dim \ker d \cap V^3 \in \{1,2\}$. Since $\dim V^2 = 1$ we have $\dim \ker d \cap (\Lambda V)^5 = \dim \ker d \cap V^3$, and so $\dim H^4 \geq 3 - \dim \ker d \cap (\Lambda V)^5$, giving $\dim H^* \geq 2(2+ \dim \ker d \cap V^3 + (3-\dim \ker d \cap(\Lambda V)^5)) = 10$. As for $(2:3,5,5,5)$ and $(2:3,5,5,7)$, since $\dim (\Lambda V)^6 = 1$, there is a non-zero closed degree 5 indecomposable, yielding $\dim H^* \geq 6$.


Moving on to formal dimension 19, the homotopy rank type $(8,8:3,15,15)$ is ruled out by noting that the degree 8 generators must be closed. For $(2,4,6:3,7,7,11)$, we are done if the degree 3 generator is closed; otherwise, the degree 4 generator and its product with the degree 2 generator are closed and non-exact, giving $\dim H^* \geq 8$. 

For $(2,6,6:3,5,11,11)$, label by $x$, $y$, $u$ the generators of degree $2,3,5$ respectively. If $y$ is closed, it and $xy$ provide two independent cohomology classes and we have $\dim H^* \geq 8$. Suppose then that $dy = x^2$. If $u$ is closed, we are done as $x^3$ is non-exact; so assume $du = x^3$, in which case $\ker d \cap (\Lambda V^{\leq 5})^7$ is spanned by $x^2y - xu$. It follows from here that $\dim \ker d \cap V^6 \geq 1$, and the product of a non-zero class in this kernel with $x$ must be closed and non-exact since $V^7 = \{0\}$; thus $\dim H^* \geq 8$. 

Next, $(4,6,6:3,7,11,11)$ is verified by noting that the degree 3 and 4 generators must be closed and non-exact, along with at least one non-zero element in $V^6$.

For the case of $(2,2,4,6:3,3,5,7,11)$, note that if $d$ is not injective on $V^3$, we have $\dim H^* \geq 10$. Suppose then that $d$ is injective on $V^3$; denoting by $\{x,x'\}$, $\{y,y'\}$, $\{z\}$ bases of $V^2$, $V^3$, $V^4$ respectively, we have $dy = ax^2 + bx'^2 + cxx'$ and $dy' = a'x^2 + b'x'^2 + c'xx'$ for some independent $(a,b,c), (a',b',c') \in \QQ^3$. As in the case of $(2,2,4,4:3,3,5,7,7)$, it follows that $(\Lambda V^{\leq 3})^5 \overset{d}{\rightarrow} (\Lambda V)^6$ has at least 3--dimensional image (note $\dim (\Lambda V^{\leq 3})^5 = 4$). If the image is 4--dimensional, i.e. the kernel is trivial, the generator in degree 4 must be closed and hence we are done. So suppose the kernel is one--dimensional and that $dz$ is non-zero. We will show that this implies the existence of a closed non-zero element in the span of $\{yy', xz, x'z\}$; combined with the fact that every element in the $4$--dimensional space $(\Lambda V^{\leq 2})^6$ is closed, and $\dim  \im ((\Lambda V)^5 \overset{d}{\rightarrow} (\Lambda V)^6) \leq 4$, we will have $\dim H^6 \geq 1$ and hence $\dim H^* \geq 10$. Now, $dz = kxy + lxy' + mx'y + nx'y'$ being closed, where $k,l,m,n \in \QQ$ are not all zero, yields the equations $$ka + la' = 0, \, \, kc + lc' + ma + na' = 0, \, \, kb + lb' + mc + nc' = 0, \, \, mb + nb' = 0. $$ If $a \neq 0$, we can rearrange our basis for $V^3$ so that $dy = x^2 + bx'^2 + cxx'$, $dy' = b'x'^2 + c'xx'$. If furthermore $b' \neq 0$, we may take $b' = 1$ and $b=0$, yielding $d(yy' - xz - cx'z) = 0$. If $a\neq 0$ and $b ' = 0$, then upon change of basis for $V^3$ we have $dy = x^2 + bx'^2$, $dy' = xx'$, and we use the above four equations to conclude $b = 0$. Then $d(yy' - xz) = 0$. The case of $b \neq 0$ is analogous to the case of $a \neq 0$. Suppose now that $c \neq 0$ and $a,b = 0$; after change of basis we have $dy = xx'$, $dy' = a'x^2 + b'x'^2$. If $b'\neq 0$, upon change of basis we have $dy = xx'$ and $dy' = a'x^2 + x'^2$. Note however that the above four equations yield $n=0$ and hence $ma' = 0$. Since $m=0$ implies $k,l,m,n = 0$ (which we are assuming is not the case), we have $a' = 0$, and $d(yy' + x'z) = 0$. If $b' = 0$, we may assume $dy = xx'$ and $dy = x^2$, giving $d(yy' + x'z) = 0$. 

In the case of $(2,4,4,4:3,5,7,7,7)$, note that $\dim H^4 \geq 3$, and so $\dim H^* \geq 10$. For $(2,4,4,6:3,3,7,7,11)$, note that if $d$ vanishes on $V^3$ we have $\dim H^4 \geq 1$, and so $\dim H^* \geq 10$. Otherwise, $\dim \ker d \cap (\Lambda V^{\leq 3})^5 = 1$ so there is a non-zero $z \in \ker d \cap V^4$. Then $z$ and its product with a degree 2 generator are closed and non-exact as $V^5 = \{0\}$, giving us $\dim H^* \geq 10$. 

Now we consider $(2,2,4,4,4:3,3,3,7,7,7)$. Suppose first that $d$ is injective on $V^3$. Then, as in the case of $(2,2,4,4:3,3,3,7,7)$ in dimension 15, we have that $\ker d \cap (\Lambda V)^5$ is two--dimensional. Therefore, there is a non-zero element in $\ker d \cap V^4$, and the product of this element with any non-zero degree two element is closed and non-exact, giving $\dim H^* \geq 12$. If $V^3 \overset{d}{\rightarrow} (\Lambda V)^4$ has trivial or one--dimensional image, then we see $\dim H^* \geq 14$ by considering only $\Lambda V^{\leq 3}$ up to degree 4. Now suppose that the image of this $d$ is two--dimensional. We can choose bases $\{x,x'\}$, $\{y,y',y''\}$ of $V^2$ and $V^3$ such that $dy = ax^2 + bx'^2 + cxx'$, $dy' = a'x^2 + b'x'^2 + c'xx'$, $dy'' = 0$, where $(a,b,c)$ and $(a',b',c')$ are linearly independent. This implies the kernel of $d$ on the six--dimensional space $(\Lambda V^{\leq 3})^5$ has dimension two or three. If the dimension is two, then $\dim \ker d \cap V^4 \geq 1$ and so $\dim H^* \geq 12$ since $\dim H^3 = 1$ and $\dim H^4 \geq 2$. If the dimension is three, then either $d$ is not injective on $V^4$ in which case we are done, or we can choose a degree four generator $z$ so that $dz = xy''$. Then $[y''z]$ is a non-zero class in $H^7$, and we have $\dim H^* \geq 12$. 

For the remaining case of $(2,4:3,3,5,5,7)$ in dimension 19, if $d$ vanishes on $V^3$ we are done, so assume that for some bases $\{x\}$, $\{y, y' \}$, $\{z\}$ of $V^2$, $V^3$, $V^4$ we have $dy = x^2$, $dy' = 0$. If $dz = 0$ we have $\dim H^* \geq 8$, so suppose $dz = xy'$. Then $\ker d \cap (\Lambda V)^6$ is spanned by $x^3$ and $yy' - xz$, and since $d(xy) = x^3$ we conclude that there is a closed element in $(\Lambda V)^5$ with a non-zero term in $V^5$ (and so by minimality it is not exact), yielding $\dim H^* \geq 8$. 


In formal dimension 20, the only remaining homotopy rank type is $(2,4,4,4:3,3,3,7,7,7)$. If $d$ vanishes on $V^3$, we are done; otherwise, choose bases $\{x\}$, $\{y,y',y''\}$ of $V^2$ and $V^3$ such that $dy = x^2$, $dy' = dy'' = 0$. We see now that $\dim \ker d \cap (\Lambda V)^5 = 2$, and so $\dim \ker d \cap V^4 \geq 1$, giving $\dim H^*\geq 10$.

\end{proof}


 \end{document}